\documentclass[11pt]{article}
\usepackage{graphicx}      

\usepackage{amsmath,amssymb}
\usepackage{amsfonts}
\usepackage{amsthm}
\usepackage{color}
\usepackage{subcaption}
\usepackage{float}

 \usepackage{newtxtext,newtxmath}

\usepackage{microtype}
\usepackage[dvipsnames]{xcolor}

\usepackage{pgfplots,tikz}
\pgfplotsset{compat=newest}
\usetikzlibrary{plotmarks}
\usetikzlibrary{arrows.meta}
\usepgfplotslibrary{patchplots}
\usepackage{grffile}
\usepackage[scale=.7]{geometry}

\usetikzlibrary{intersections,calc,patterns,angles,quotes}
\usetikzlibrary{decorations.shapes,arrows,calc,decorations.markings,shapes,decorations.pathreplacing,shapes.arrows,arrows.meta	}
\usepackage{tikz-3dplot}

\newtheorem{theorem}{Theorem}
\newtheorem{lemma}{Lemma}

\newcommand{\R}{\mathbb R}
\newcommand{\grad}{\operatorname{grad}}

\newcommand{\U}{\mathcal U}

\newcommand{\X}{\mathcal X}
\newcommand{\cA}{\mathcal A}
\newcommand{\dk}[1]{\frac{\partial #1}{\partial x_k}}
\newcommand{\dl}[1]{\frac{\partial #1}{\partial x_l}}
\newcommand{\dx}[1]{\frac{\partial #1}{\partial x}}

\title{The Affine Geometric Heat Flow and Motion Planning for Dynamic Systems}%
        \date{}
        \author{Shenyu Liu, Yinai Fan and Mohamed Ali Belabbas\thanks{Shenyu Liu, Yinai Fan  and Mohamed Ali Belabbas are with the Coordinated Science Laboratory, University of Illinois, Urbana-Champaign.
        {\tt\small sliu113, yfan17, belabbas@illinois.edu}}}

\begin{document}

\maketitle
\begin{abstract}             
We present a new method for motion planning for control systems. The method aims to provide a natural computational framework in which a broad class of motion planning problems can be cast; including problems with holonomic and non-holonomic constraints, drift dynamics, obstacle constraints and constraints on the magnitudes of the applied controls. The method, which finds its inspiration in recent work on the so-called geometric heat flows and curve shortening flows, 
relies on a hereby introduced partial differential equation, which we call the affine geometric heat flow, which evolves an arbitrary differentiable path joining initial to final state in configuration space to a path that meets the constraints imposed on the problem. From this path, controls to be applied on the system can be extracted. We provide conditions guaranteeing that the controls extracted will drive the system arbitrarily close to the desired final state, while meeting the imposed constraints and illustrate the method on three canonical examples.
\end{abstract}

\section{Introduction}

Given a control system \begin{equation}\label{eq:sys0} \dot x = f(x,u)\end{equation} evolving on a differentiable manifold $M$, and two points $x_i, x_f \in M$, the motion planning problem in time $T>0$ is to find a control $u^*(t)$ that {\it steers} the system from $x_i$ to $x_f$ in $T$ units of time, i.e. so that the solution $x^*(t)$ of Eq.~\eqref{eq:sys0} with $u=u^*$ and $x^*(0)=x_i$ yields $x^*(T)=x_f$.
Due to its ubiquity in control applications ranging from robotics to autonomous wheeled vehicles, motion planning has been widely studied~\cite{Laumond1998,lavalle_2006} and a host of methods have been developed.

One of the control papers in which the issue of motion planning for non-holonomic systems was clearly delineated is the seminal paper~\cite{Brockett1982}, where finding a subRiemannian geodesic can easily be seen as solving a non-holonomic motion planning problem. For a more recent survey of this line of work, we refer to the recent monograph~\cite{Jean2014}. A common approach to non-holonomic motion planning is to use sinusoidal control function to, roughly speaking, generate the ``Lie bracket'' directions. See for example~\cite{li2012nonholonomic}, and~\cite{WANG2019114} for a very recent work on how oscillations can be used for orientation control in $SO(3)$. This idea is also used in derivative-free optimization~\cite{durr2013lie, michalowsky2017distributed}.

The major difficulties that can arise in motion planning problems are: 1. the nonholonomic character of the dynamics, 2. the presence of a drift term, 3. the presence of constraints on the control inputs/states. From a theoretical point of view, the Chow-Rashevski theorem provides us with conditions under which a non-holonomic system is controllable~\cite{}, but in the latter two cases, no such conditions is known: the general case of controllability of non-linear systems with drift is still an open problem. Nevertheless, for some specific nonlinear systems with drift, path planning or control algorithms are given in \cite{Luca1995,611856,788533}.

While most of the methods mentioned above focused on addressing the first difficulty, we propose here an approach that can address all three. More precisely, we propose a new variational method for motion planning. The novelty of the work lies in an extension of the so-called {\it geometric heat flow}, see below or~\cite{curvebook2001} for its use in motion planning, to encompass dynamics with constraints and drift. We term the resulting flow the {\it affine-geometric heat flow}.  The method works by ``deforming'' an arbitrary path between $x_i$ and $x_f$ into an almost feasible trajectory for the system, from which we can extract the controls $u^*$ that drive the system from $x_i$ to $x_f$ approximately. Using a variational approach, we provide a proof of convergence of the method as well.

\section{Preliminaries}\label{sec:prelim}

\noindent {\bf Notation:} With a slight abuse of notation, we define for $G(x) \in \R^{n \times n}$, $f,g \in \R^n$ $$\left[ f\left(\dx{G}\right)g\right]:=\begin{pmatrix} f^\top\frac{\partial G}{\partial x_1}g \\ \vdots \\ f^\top \frac{\partial G}{\partial x_n} g\end{pmatrix} \in \R^{n};$$ i.e. $\left[ f\left(\dx{G}\right)g\right]$ is the vector in 
$\R^n$ with $i$th entry $f^\top \frac{\partial G}{\partial x_i} g$, where $\frac{\partial G}{\partial x_i}$ is the $n \times n$ matrix with $kl$ entry $\frac{\partial G_{kl}}{\partial x_i}$. 
We furthermore use the notation \begin{equation}\label{eq:not2} (f \cdot G):=\left(\sum_{l=1}^n f_l \frac{\partial G_{ij}}{\partial x_l}\right)_{ij} \in \R^{n \times n}.\end{equation}

\noindent{\bf Trajectory planning problem} Consider a control system evolving in $M=\R^n$. We refer to $M$ as the configuration space. Note that $M$ can more generally be a $C^2$-differentiable manifold. We consider here dynamics which is {\it affine in the control}, that is the system evolves according to
\begin{equation}\label{sys}
    \dot x=F_d(x)+F(x)u
\end{equation}
for each $x\in\R^n$, $F_d(x)$ is a vector representing the drift dynamics when in state $x$; the columns of $F(x)\in\R^{n\times m}$ are the admissible control directions. The method we propose can also be applied to the more general dynamics described in Eq.~\eqref{eq:sys0}, we briefly discuss how towards the end of the paper. We make the following {\it assumption}: 
\begin{description}
    \item[Assumption A]
both $F_d(x), F(x)$ are assumed to be at least $C^2$, Lipschitz with constants $L_1, L_2$ respectively, and we assume that $F(x)$ is of {\it constant} rank almost everywhere in $\R^n$.  
\end{description}
Note that these assumptions can be weakened at the expanse of longer analysis. We focus here on the  the case  $n\geq m$; that is on potentially {\it under-actuated} dynamics.

Recall that $x_i, x_f \in M$ are the desired initial and final states respectively, and $T >0$ is a fixed time allowed to perform the motion.
The set of {\bf admissible controls} is $\U:=L_2([0,T]\to \R^m)$, that is square integrable functions defined over the interval $[0,T]$. We set $$\X:=\{x(\cdot)\in AC([0,T]\to\R^n):x(0)=x_i,x(T)=x_f\},$$ the space of {\it absolutely continuous} $\R^n$-valued functions with start- and end-values  $x_i$ and $x_f$ respectively. 

We call  any $x(\cdot)\in\X$ an {\bf admissible solution} if there exists $u \in \U$ so that the generalized derivative $\dot x(t)$ of $x(\cdot)$ satisfies~\eqref{sys}. Denote by $\X^*\subseteq\X$ the set of admissible solutions. The {\bf trajectory planning problem} (from $x_i$ to $x_f$ with time $T$) is {\bf feasible} if $\X^*\neq\emptyset$. All open sets in $\X$ are with respect to the natural norm $\Vert x\Vert_{AC}:=\int_0^T\left(|x(t)|+|\dot x(t)|\right)dt.$

We additionally introduce the space of {\it continuous controls} $\U':=C^0([0,T]\to\R^m)$, to which correspond differentiable trajectories $\X':=\{x(\cdot)\in C^1([0,T]\to\R^n):x(0)=x_i,x(T)=x_f\}$. It is well-known that $\U'$ is a dense subspace in $\U$ with respect to the $L_2$ norm and that $\X'$ is a dense subspace in $\X$ with respect to the $\Vert\cdot\Vert_{AC}$ norm. Working over $\X'$ instead of $\X$ allows us to ``smoothly deform" a curve, which we will explain later.

\section{A flow for motion planning}\label{sec:motion_planning}

We now briefly introduce the geometric framework we use to cast and solve the trajectory planning problem and then introduce the main object introduced in this work: the {\it affine geometric heat flow}. We refer to our earlier work~\cite{7963599} for a more detailed presentation and more examples about the general framework. We rely  on differentiable {\bf homotopies} of curves (i.e., differentiable ``deformations'' of a curve), where the variable $s$ is the homotopy parameter; precisely, we use $x(t,s): [0,T]\times[0,s_{\max})$ where for {\it each $s$ fixed}, $x(\cdot,s) \in \X'$. 

Let $G(x)$ be a positive definite matrix defined on $M$, which we refer to as the Riemannian {\bf metric}. We denote by $\nabla$ the {\bf Levi-Civita} connection of $G(x)$~\cite{Jost2011}, and by $\nabla_{f}g$ the covariant derivative of the vector field $g$ along the vector field $f$. Recall that if $a(t) = \sum_{k=1}^n a_k(t) e^k$, where $a_i(t)$ are real numbers and $\|e^k\}$ basis vectors,  is a vector field along a curve $x(t)$, and $g$ is a vector field defined in a neighborhood of $x(t)$, then $\nabla_{a}g := \frac{da}{dt} + \sum_{i,j,k} \Gamma_{ij}^k a_ig_j e^k.$ 

The so-called {\bf geometric heat flow (GHF)}  is a parabolic partial differential equation, which evolves a curve with fixed end-pints toward a curve of minimal length: namely, given a Riemannian metric and an associated Levi-Civita connection, the GHF is the PDE 
\begin{equation}\label{eq:ghf}
    \frac{\partial x(t,s)}{\partial s} = \nabla_{\dot x(t,s)}\dot x(t,s).
\end{equation} where $\dot x(t,s):=\frac{\partial x(t,s)}{\partial t}$ and $x(0,s)=x_i, x(T,s)=x_f$ are fixed.
We refer the reader to~\cite{Jost2011} for a proof that this PDE yields a curve of minimal length. For applications of this flow to motion planning problems, and some illustrations of its solutions, we refer~\cite{7963599}

Taking inspiration from the GHF, we introduce here what we term the {\bf affine geometric heat flow (AGHF)}:
\begin{equation}\label{eq:aghf}
    \frac{\partial x(t,s)}{\partial s} = \nabla_{\dot x(t,s)}\left(\dot x(t,s)-F_d\right)+r(t,s)
\end{equation} 
where \begin{small}$$r(t,s)= G^{-1}\left( \left(\frac{\partial F_d}{\partial x}\right)^\top G (\dot x-F_d) + \frac{1}{2}\left[(\dot x -F_d) \left(\dx{G}\right)F_d\right]\right)$$ \end{small}and $F_d= F_d(x(t,s))$, $G=G(x(t,s))$, etc.  The first term is the covariant derivative of $\dot x(t,s)-F_d(x)$ in the direction $\dot x(t,s)$: it behaves similarly to the GHF, but takes into account the existence of a ``drift'' vector field $F_d(x)$. The second term, $r(t,s)$, can be interpreted as follows: denote by $\grad k (x,a(t))$ the gradient vector field, for the metric $G$, of the parametric function $k(x,a(t)) = \left< a(t)-F_d(x), F_d(x)\right>$, where the inner product $\left<\cdot,\cdot\right>$ is given by $G$ as well, and $a(t)$ is considered a parameter. Then we can show that $r(t,s) = \grad k(x(t,s),\dot x(t,s))$. One can justify the use of $r(t,s)$ as follows: it ``deforms'' $x(t,s)$ in the direction that aligns $\dot x$ with the drift vector field $F_d(x)$, thus rendering the path admissible for said drift. Due to space constraint, we postpone a lengthier justification for the form of the AGHF, but Lemma~\ref{lem:decreasing V} below provides a proof of convergence and additional insights, and we furthermore show in Sec.~\ref{sec:examples} that it does indeed converge to admissible paths on various examples. Note that when $F_d = 0$, the AGHF reduces to the GHF.

We say a solution $x^*(t)$ is a {\bf steady state} solution of~\eqref{eq:aghf} if its RHS is $0$ almost everywhere when evaluated on $x^*(t)$.

Note that when $F_d=0$, then $r(t,s) \equiv 0$ and the AGHF reduces to the GHF. Some intuition about the form of this PDE can be gathered from Lemma~\ref{lem:decreasing V} below.
The AGHF is a system of {\it parabolic PDEs}, and to solve it we need two {\bf boundary} conditions
\begin{equation}\label{bc}
    x(0,s)=x_i,x(T,s)=x_f\quad\forall s\geq 0
\end{equation}
and an {\bf initial} condition
\begin{equation}\label{ic}
    x(t,0)=v(t),\quad t\in[0,T]
\end{equation}
for some $v(\cdot)\in\X'$.

\noindent{\bf The method.}  We first describe the steps of the method we propose to perform trajectory planning for systems \eqref{sys} and provide below some converge guarantees. The method can be summarized through the following four steps. We are given $F_d(x)\in \R^n$, $F(x) \in \R^{n \times n}$, $x_i \in \R^n$ and $x_f \in \R^n$.
\begin{description}
    \item[Step 1:] Find a {\it bounded} $n\times (n-m)$ $x$-dependent matrix $F_c(x)$, differentiable in $x$, such that 
    \begin{equation}\label{def:bar_F}
    \bar F(x):=\begin{pmatrix}F_c(x)|F(x)
    \end{pmatrix} \in \R^{n \times n}
    \end{equation}
    is invertible for all $x\in\R^n$. The matrix $F_c(x)$ can be obtained using, e.g., the Gram-Schmidt procedure.
    
    \item[Step 2:] Evaluate 
    \begin{equation}\label{def:G}
    G(x):=(\bar F(x)^{-1})^\top D\bar F(x)^{-1}
    \end{equation}
    where $D:=\mbox{diag}(\underbrace{\lambda,\cdots,\lambda}_{n-m},\underbrace{1,\cdots,1}_{m})$ for some large $\lambda>0$ (we discuss below what large means in this context).

    \item[Step 3:]  
    Solve the AGHF~\eqref{eq:aghf} with boundary conditions~\eqref{bc} and initial condition~\eqref{ic}. Denote the solution by $x(t,s)$;

    \item[Step 4:] Evaluate
   {\small \begin{equation}\label{control_extraction}
         u(t):=\begin{pmatrix}
        0&I_{m\times m}
        \end{pmatrix}\bar F(x(t,s_{\max}))^{-1}(\dot x(t,s_{\max})-F_d(x(t,s_{\max}))
   \end{equation} }
\end{description}

{\bf Output:} The control $u(t)$ obtained in~\eqref{control_extraction} yields, when integrating~\eqref{sys}, a trajectory $\tilde x(t)$, which is our solution to the trajectory planning problem. We call it {\bf integrated path}.

It is observed that $F_c$ does not need to depend on the drift $F_d$ and there is no orthogonality requirement between $F$ and $F_c$ either, which gives much freedom in the construction of $F_c$ and hence in many cases the form of $\bar F=(F_c|F)$ and consequently $G$ are quite simple. We show how to apply the method to incorporate constraints on the controls in Sec.~\ref{sec:constraints} and we will provide examples of the application of the method in Sec.~\ref{sec:examples}.

\subsection{On the convergence of the AGHF.}

The main new ingredients our method proposes is the definition of the Riemannian inner product in Step 2 above, and the definition of the AGHF. Roughly speaking, for this inner product, short paths are admissible paths, and we refer to our earlier work~\cite{LBMotionSktechArxiv2019} for a longer justification of the use of the inner product defined. 

The AGHF is a nonlinear set of PDEs, and thus the existence of a solution is not guaranteed a priori even for short time. We provide here an analysis and convergence guarantees.

To this end, define  the Lagrangian $L$ by 
    \begin{equation}\label{def:L}
    L(x,\dot x)=\frac{1}{2}(\dot x-F_d(x))^\top G(x)(\dot x-F_d(x))
    \end{equation} 
Given $L$,  the {\it action functional} is defined on $\X$ as 
\begin{equation}\label{def:V}
    \cA(x(\cdot)):=\int_0^TL(x(t),\dot x(t))dt.
\end{equation}

\begin{lemma}\label{lem:decreasing V}
Let $x^*(t)$ be a steady-state solution of the AGHF~\eqref{eq:aghf}. Then $x^*(t)$ is an extremal curve for $\cA$ in \eqref{def:V}. Furthermore, $\cA$ decreases along the solutions of the AGHF; i.e. if $x(t,s)$ is such a solution, then $\frac{d}{ds}\cA(x(\cdot,s)) \leq 0$, and equality holds only if $x(\cdot,s)$ is an extremal curve for $\cA$.

\end{lemma}
\begin{proof}
Due to space constraints, we only sketch the proof. We first show that a steady-state solution of the AGHF satisfies the Euler-Lagrange equation of $\cA$ and hence it is an extremal curve. To this end, write $f:=F_d$ and $L=\frac{1}{2} (\dot x_i - f_i) G_{ij}(\dot x_j-f_j)$, where we adopt the Einstein convention of summing over repeated indices. 
We have
\begin{align*}
    \frac{\partial L}{\partial x_k} &= -\dk{f_i} G_{ij}(\dot x_j -f_j) + \frac{1}{2} (\dot x_i -f_i) \dk{G_{ij}} (\dot x_j - f_j),\\
     \frac{\partial L}{\partial \dot x_k} &=(\dot x_i-f_i) G_{ik},\\
     \frac{d}{dt}  \frac{\partial L}{\partial \dot x_k} &= (\ddot x_i -\dl{f_i} \dot x_l )G_{ik}+(\dot x_i-f_i) \dl{G_{ik}}\dot x_l.
\end{align*}

The Euler-Lagrange equations can be written, after some rearrangments and using the notation introduced in Sec.~\ref{sec:prelim}, as
\begin{multline}\label{eq:EL1}
0=\frac{d}{dt}  \frac{\partial L}{\partial \dot x}-   \frac{\partial L}{\partial x} =( \dot x \cdot G)
(\dot x - f) + G (\ddot x -\frac{\partial f}{\partial x} \dot x)\\  + \left(\frac{\partial f}{\partial x}\right)^\top G (\dot x -f)-\frac{1}{2} \left[ (\dot x -f) \left(\dx{G}\right) (\dot x-f)\right]
\end{multline}
After some algebraic computations, which we omit here due to space constraints, we obtain that \eqref{eq:aghf} is equivalent to the following PDE: 
\begin{equation}\label{eqn:EL2}
    \frac{\partial x}{\partial s}=G^{-1}\left(\frac{d}{dt}\frac{\partial L}{\partial \dot x}-\frac{\partial L}{\partial x}\right),
\end{equation}
which completes the first part of the proof. Next, apply first order variation approximation on \eqref{def:V} and integration by parts, we have 
\begin{equation*}
    \cA(x(\cdot,s+\delta))=\cA(x(\cdot,s))+\delta\left(\left.\frac{\partial x(t,s)}{\partial s}^\top\frac{\partial L}{\partial \dot x}\right|_0^T\right.
    +\left.\int_0^T\frac{\partial x(t,s)}{\partial s}^\top\left(\frac{\partial L}{\partial x}-\frac{d}{dt}\frac{\partial L}{\partial \dot x}\right)dt\right)+o(\delta).
\end{equation*}
The evaluated term $\left.\frac{\partial x}{\partial s}^\top\frac{\partial L}{\partial \dot t}\right|_0^T$ vanishes because of our boundary conditions \eqref{bc}. Taking the limit $\delta\to 0$ and plug \eqref{eqn:EL2} in,
\begin{align*}
\frac{d \cA(x(\cdot,s))}{d s}&=\lim_{\delta\to 0}\frac{V(x(\cdot,s+\delta))-V(x(\cdot,s))}{\delta}\\
&=\int_0^T\frac{\partial x(t,s)}{\partial s}^\top\left(\frac{\partial L}{\partial x}-\frac{d}{dt}\frac{\partial L}{\partial x_t}\right)dt\\
&=-\int_0^T \frac{\partial x(t,s)}{\partial s}^\top G(x) \frac{\partial x(t,s)}{\partial s} dt
\end{align*}
By definition~\eqref{def:G}, $G$ is positive definite so $\frac{d \cA(x(\cdot,s))}{d s}\leq 0$ and equality holds if and only if $\frac{\partial x}{\partial s}=0$ almost everywhere, i.e., $x(\cdot,s)$ is an extremal curve for $\cA$.
\end{proof}

\subsection{On convergence guarantees for motion planning}

We now show that the control extracted in Step 4 of the method will drive the system arbitrarily close to the desired final state, {\it provided} that a solution to the motion planning problem exist (which, as we mentioned earlier, is in general an open problem for systems with drift), that the trajectory with which we initialize the system is well-chosen. We will see in the examples of Sec.~\ref{sec:examples}, that an arbitrary choice of initial condition very often yields a convergent solution.

\begin{theorem}
\label{thm:main}
Consider the system \eqref{sys} and let $x_i, x_f \in \R^n$. Assume that the motion planning problem from $x_i$ to $x_f$ is feasible (i.e. $\X^*$ is non-empty) and that Assumption A above is met. Then there exists $C>0$ such that for any $\lambda>0$, there exists an open set $\Omega_\lambda\subseteq\X'$ (with respect to $\Vert\cdot\Vert_{AC}$) so that as long as the initial curve $v\in\Omega_\lambda$, The integrated path $\tilde x(t)$ from our algorithm with sufficiently large $s_{\max}$ has the property that 
\begin{equation}\label{endpoint_convergence}
|\tilde x(T)-x_f|\leq \sqrt{\frac{3TMC}{\lambda}}\exp{\left(\frac{3T}{2}(L_2^2T+L_1^2C)\right)}.
\end{equation}
\end{theorem}
We see that as $\lambda \to \infty$, we drive the system closer to the desired final state $x_f$.

\begin{proof}
Since the path planning problem is feasible, there exists $x^*(\cdot)\in\X^*$ so we have $\dot x^*=F_d(x^*)+F(x^*)u^*$ for some $u^*(\cdot)\in \U$. Plug this $x^*$ into \eqref{def:L} and we have $L(x^*(t),\dot x^*(t))=|u^*(t)|^2$. Pick $C>\cA(x^*)=\int_0^T|u^*(t)|^2dt$. Note $C$ is independent of $\lambda$. Denote $\Omega_\lambda^{\text{AC}}:=\{x\in\X:V(x)<C\}$. $\Omega_\lambda^{\text{AC}}$ is not empty because it at least contains $x^*$; in addition, because $\cA$ is continuous over $\X$ with respect to $\Vert\cdot\Vert_{AC}$, $\Omega_\lambda^{\text{AC}}$ is open. Since $\X'$ is dense in $\X$, $\Omega_{\lambda}':=\Omega_\lambda^{\text{AC}}\cap \X'$ is open as well. From Lemma~\ref{lem:decreasing V} we know that $\cA(x(\cdot,s))$ is non-increasing so $\Omega_\lambda'$ is invariant. Let $\Omega_\lambda$ be the region of attraction to $\Omega_\lambda'$; that is, $\Omega_\lambda'\subset\Omega_\lambda\subset \X'$ and all AGHF solutions $x(\cdot,s)$ derived from \eqref{eq:aghf} with any initial condition $v\in\Omega_\lambda$ will converge to the invariant set $\Omega_\lambda'$ when $s$ increases. Consequently when $s_{\max}$ is sufficiently large, $\cA(x(\cdot,s_{\max}))\leq C$.

Define the curve $x(t):=x(t,s_{\max})$ and for each $t\in[0,T]$ let $u(t)\in\R^m,u_c(t)\in \R^{n-m}$ be given by
\begin{equation}\label{eqn:full_u}
\begin{pmatrix}
u_c(t)\\u(t)
\end{pmatrix}=\bar F(x(t))^{-1}(\dot x(t)-F_d(x(t))).
\end{equation}

Plug it into \eqref{def:L} and we have
\[
C\geq \cA(x)=\int_0^TL(x(t),\dot x(t))dt=\int_0^T|u(t)|^2+\lambda|u_c(t)|^2dt.
\]
\begin{equation}\label{inequality_1}
    \Rightarrow \int_0^T|u(t)|^2dt\leq C,\quad\int_0^T|u_c(t)|^2dt\leq \frac{C}{\lambda}.
\end{equation}

Comparing \eqref{eqn:full_u} with \eqref{control_extraction}, we see that the extracted control is exactly $u$; in other words, the integrated path is given by 
\[
\tilde x(0)=x_i,\quad \dot{\tilde x}=F_d(\tilde x)+F(\tilde x)u.
\]
Define the error $e(t):=x(t)-\tilde x(t)$. Then
\[
e(0)=0,\ \dot e=(F_d(x)-F_d(\tilde x))+(F(x)-F(\tilde x))u+F_c(x)u_c.
\]
Therefore we have
\begin{equation*}
e(t)=\int_0^t(F_d(x(\tau))-F_d(\tilde x(\tau)))+(F(x(\tau))-F(\tilde x(\tau)))u(\tau)+F_c(x(\tau))u_c(\tau)d\tau 
\end{equation*}
Square the norm of $e(t)$ and apply Cauchy-Schwartz inequality,
\begin{equation*}
|e(t)|^2 \leq t\int_0^t(|F_d(x(\tau))-F_d(\tilde x(\tau))|
 +|(F(x(\tau))-F(\tilde x(\tau)))u(\tau)|+|F_c(x(\tau))u_c(\tau)|)^2d\tau
\end{equation*}
Use power mean inequality, the square of the sum of 3 terms inside integral is no larger than 3 times the sum of the square of each individuals,
\begin{equation*}
|e(t)|^2\leq 3t\int_0^t|F_d(x(\tau))-F_d(\tilde x(\tau))|^2
+|(F(x(\tau))-F(\tilde x(\tau)))u(\tau)|^2+|F_c(x(\tau))u_c(\tau)|^2d\tau
\end{equation*}
Owing to the fact that $ F_d,  F$ are globally Lipschitz and $ F_c $ is globally bounded, we conclude that
\begin{small}\begin{align*}
|e(t)|^2\leq 3t\int_0^t(L_2^2+L_1^2|u(\tau)|^2)|e(\tau)|^2d\tau+3t\int_0^tM|u_c(\tau)|^2d\tau
\end{align*}\end{small}
Next, by Gr\"onwall inequality and the fact that $3t\int_0^t(L_2^2+L_1^2|u(\tau)|^2)|e(\tau)|^2d\tau$ is a non-decreasing function of $t$,
\begin{small}\[
|e(t)|^2\leq \left(3t\int_0^tM|u_c(\tau)|^2d\tau\right)\exp\left(3t\int_0^t(L_2^2+L_1^2|u(\tau)|^2)d\tau\right)\\
\]\end{small}
Finally, substitute in the inequalities from \eqref{inequality_1}, we conclude that $|e(t)|^2\leq \frac{3tMC}{\lambda}\exp{\left(3t(L_2^2t+L_1^2C)\right)}$. Thus $|\tilde x(T)-x_f|=|e(T)|\leq \sqrt{\frac{3TMC}{\lambda}}\exp{\left(\frac{3T}{2}(L_2^2T+L_1^2C)\right)}$ as was to be proved.\qed
\end{proof}

\section{Beyond affine control systems}\label{sec:constraints}

As we  already discussed in our earlier work~\cite{LBMotionSktechArxiv2019}, constraints in states can be handled by multiplying the inner product matrix $G(x)$ by an appropriately defined barrier function $b(x)$. Now we show that being able to handle constraints and drift allows us to handle input constraints and non-affine in the control dynamics, i.e., systems of the type
\begin{equation}\label{sys:nonlinearwithconstraints}
\begin{aligned}\dot x&=f(x,u)\\
 l(x,u(t))&\geq 0,\quad\forall t\in[0,T],
 \end{aligned}
\end{equation} with $l$ a differentiable function, at the expanse of using controls that are differentiable almost everywhere. 
For example, common constraints such as magnitude bounds on the controls $|u| \leq u^{\max}$, for a given $u^{\max}$,   can be implemented  by setting $l(u)=u^{\max}-|u|\geq 0$. Define $\dot u=v$. Denote the {\bf augmented state} $y=\begin{pmatrix}x\\u\end{pmatrix}\in\R^{n+m}$, then
\begin{equation}\label{sys:aug}
    \dot{y}=\begin{pmatrix}
    \dot x\\\dot u
    \end{pmatrix}=\underbrace{\begin{pmatrix}f(x,u)\\0
    \end{pmatrix}}_{F_d}+\underbrace{\begin{pmatrix}0\\I_{m\times m}
    \end{pmatrix}}_{F}v,
\end{equation}
which is a system with affine control and drift, same as \eqref{sys}. To implement the constraints in $u$, we pick the barrier function $b(y):=1+\frac{1}{l(x,u)}$ and multiply it to $G$ in \eqref{def:G}:
\[
G(y):=b(y)(\bar F(y)^{-1})^\top D\bar F(y)^{-1}
\]
Notice that if we pick $F_c=\begin{pmatrix} 0&I_{n\times n}\end{pmatrix}^\top$, we will simply have $\bar F=(F_c|F)=I_{(n+m)\times(n+m)}$. Thus $G(y)=b(y)D$. In addition, \eqref{def:L} gives
\begin{align*}
L( y,\dot{y})&=b(\bar y)(\dot y-F_d(y))^\top G(y)(\dot y-F_d(y))\\
&=b(\bar y)(\lambda|\dot x-f(x,u)|^2+|\dot u|^2)
\end{align*}
We can argue that the AGHF will yield an admissible control meeting the constraints as follows:  when a curve $y(t)$ is  close to the boundary of obstacles defined by the $l$, $l$ is close to $0$ and thus $b(y)$ and $L$ are large. Since the AGHF minimizes $\cA=\int_0^TLdt$, the curve will be deformed away from the obstacles and hence $u$ will meet the input constraints. During the implementation of this algorithm in Step 3,  since $u_i$'s are now states of the augmented system, initial conditions and boundary conditions need to be provided.
\section{Applications and examples}\label{sec:examples}
We illustrate our method on three canonical motion planning problems, showcasing how it handles the different issues that can arise. The basic system we consider is a unicycle rolling on the plane without slipping, as depicted in Fig.~\ref{fig:unicycle}. The kinematics is given by the differential equations
\begin{equation}\label{eqn:unicycle}
\underbrace{\begin{pmatrix}
\dot{q}_x \\ \dot{q}_y \\ \dot{\theta} 
\end{pmatrix}}_{\dot{x}} = 
\underbrace{\begin{pmatrix}
\cos{\theta} \\ \sin{\theta} \\ 0
\end{pmatrix}}_{f_1} u_1+ 
\underbrace{\begin{pmatrix}
0 \\ 0 \\ 1
\end{pmatrix}}_{f_2} u_2. 
\end{equation}
We refer to our earlier work~\cite{7963599} for basic motion planning tasks for this system.

\begin{figure}
           \centering
 \includegraphics{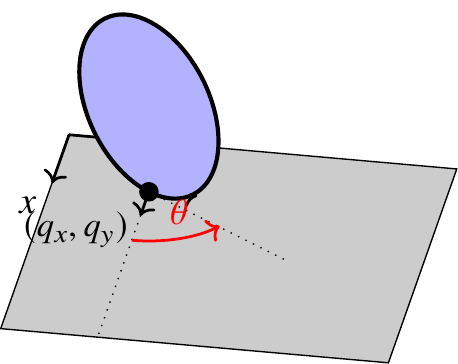}

\caption{\small The unicycly is a 3dof system with configuration variables $(q_x,q_y)$, describing the position of the center of the wheel,  and  $\theta$ describing the orientation of the wheel with respect to the $x$-axis.}
\label{fig:unicycle}
\end{figure}

\subsection{Unicycle of constant linear velocity}
We first consider the model of a planar unicycle with unit {\it constant linear velocity}. In this case we have $u_1\equiv 1$ and hence \eqref{eqn:unicycle} becomes $\dot x=f_1+f_2u$, where $f_1$ is now the drift. In other words, even when $u=0$, $\dot x\neq 0$. The control $u$ only allows to steer the unicycle. 

Following our method, we can pick $F_c=\begin{pmatrix}e_1 &e_2\end{pmatrix}$ so that $\bar F$ is the identity matrix. Hence according to Step 2 we have $G=\mbox{diag}(\lambda,\lambda,1)$. The corresponding Lagrangian~\eqref{def:L} is
\[
L(x,\dot x)=\lambda(\dot q_x-\cos(\theta))^2+\lambda(\dot q_y-\sin(\theta))^2+\dot\theta^2
\]
The boundary conditions are set to  $x_i=\begin{pmatrix}0&0&0\end{pmatrix}^\top$ and $x_f=\begin{pmatrix}0&1&0\end{pmatrix}^\top$, which are the boundary conditions of the so-called ``parallel parking'' problem. We obtain the results shown in Fig.~\ref{fig:Dubin}.

\begin{figure}
\tikzset{every picture/.style={scale=0.58}}
\begin{subfigure}{.5\columnwidth}
  \centering
	\includegraphics{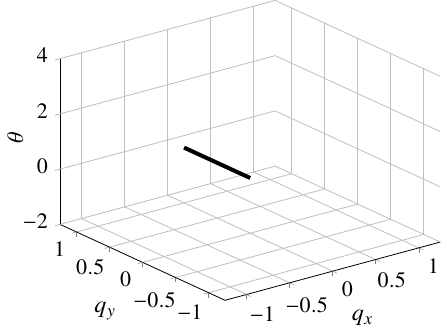}
  \caption{s=0}
  \label{fig:dubin_s0_3d}
\end{subfigure}%
\begin{subfigure}{.5\columnwidth}
  \centering
  \includegraphics{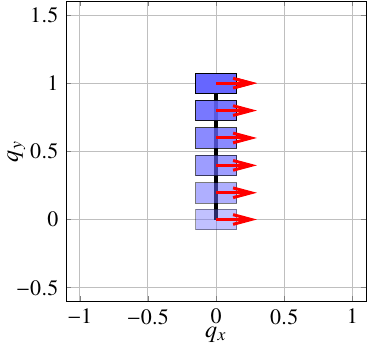}
  \caption{s=0}
  \label{fig:dubin_s0_2d}
\end{subfigure}
\begin{subfigure}{.5\columnwidth}
  \centering
    \includegraphics{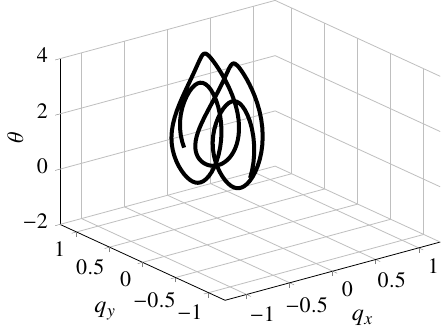}
  \caption{s=10}
  \label{fig:dubin_s10_3d}
\end{subfigure}%
\begin{subfigure}{.5\columnwidth}
  \centering
  \includegraphics{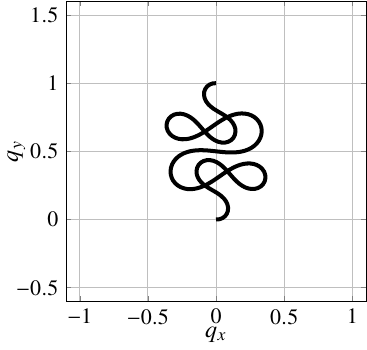}
  \caption{s=10}
  \label{fig:dubin_s10_2d}
\end{subfigure}
\begin{subfigure}{.5\columnwidth}
  \centering
  \includegraphics{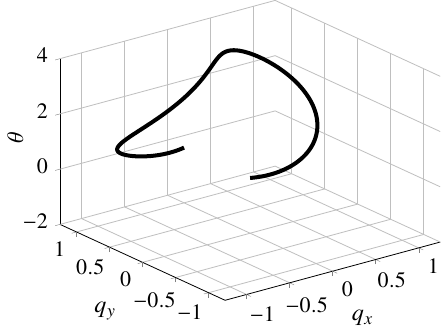}
  \caption{s=500}
  \label{fig:dubin_s500_3d}
\end{subfigure}%
\begin{subfigure}{.5\columnwidth}
  \centering
  \includegraphics{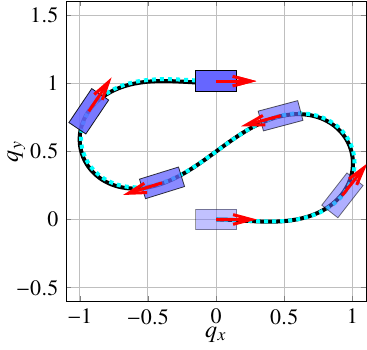}
  \caption{s=500}
  \label{fig:dubin_s500_2d}
\end{subfigure}
\centering
\caption{\small Unicycle trajectory with $\lambda=1000, T=5$. Left column: paths in 3D state space; right column: corresponding $(q_x,q_y)$-plane projected view. The unicycle follows the black solid curve and moves from the position with the lightest blue color to positions with darker blue colors gradually, with its orientation and magnitude of linear velocity at each snapshot indicated by the red arrow.}
\label{fig:Dubin}
\end{figure}
We see that the initial sketch of straight line $x(t,0)=v(t)=(0,t,0)^\top$ in Fig.~\ref{fig:dubin_s0_3d} and Fig.~\ref{fig:dubin_s0_2d} cannot be followed by the unicycle as such a path violates the non-slip constraints and does not follow the drift dynamics. Fig.~\ref{fig:dubin_s10_3d} and Fig.~\ref{fig:dubin_s10_2d} show the curve $x(t,s)$ with $s=10$. Fig.~\ref{fig:dubin_s500_3d} and Fig.~\ref{fig:dubin_s500_2d} shows $x(t,s_{\max})$ with $s_{\max}=500$. The integrated trajectory (cyan dotted line) generated by using the extracted control (as in Step 4) is very close to $x(t,s_{\max})$ and drives the unicycle close to $x_f$ with very high precision.

\subsection[width=0.5\textwidth]{Dynamic unicycle}

We now consider the unicycle with inertia; the acceleration of the unicycle is proportional to the applied torque following Newton's second law. To model this system in the form of~\eqref{sys}, we add two states to the unicycle configuration: $u_1$ and $u_2$, representing the linear and angular velocity. The controls  $v_1,v_2$,  because they act on the accelerations $\dot u_1, \dot u_2$, can be thought of as torques.
\begin{equation}\label{sys:dynamic_unicycle}
\underbrace{\begin{pmatrix}
\dot{q}_x \\ \dot{q}_y \\ \dot{\theta} \\ {\dot{u}}_1 \\ {\dot{u}}_2
\end{pmatrix}}_{\dot{x}} = 
\underbrace{\begin{pmatrix}
u_1\cos{\theta} \\ u_1\sin{\theta} \\ u_2 \\ 0 \\ 0 
\end{pmatrix}}_{F_d} + 
\underbrace{\begin{pmatrix}
0 & 0 \\ 0 & 0 \\ 0 & 0 \\ 1 & 0 \\ 0 & 1
\end{pmatrix}}_{F}
\underbrace{\begin{pmatrix}
v_1 \\ v_2
\end{pmatrix}}_{v}
\end{equation}
Similar to the previous case, we can take $\bar F$ to be the  identity matrix and $G=\mbox{diag}(\lambda,\lambda,\lambda,1,1)$. Consequently,
\begin{equation}\label{L_dynamic_unicycle}
L(x,\dot x)=\lambda\left((\dot q_x-u_1\cos(\theta))^2+(\dot q_y-u_1\sin(\theta))^2\right.
\left.+(\dot\theta-u_2)^2\right)+\dot u_1^2+\dot u_2^2
\end{equation}
We set the boundary condition to $x_i=(0, 0, 0, 0, 0)^\top$ and $x_f=(0, -1, 0, 0, 0)^\top$. The boundary values for $u_1$ and $u_2$ are 0, meaning the unicycle starts and ends with 0 velocities. We use  a partial sinusoid $v(t)=(\sin(2\pi t),-t,0,0,0)$ as the initial sketch $x(t,0)$, shown in Fig.~\ref{fig:DynamicCar_s0_3d} and Fig.~\ref{fig:DynamicCar_s0_2d}. Following the remaining steps of the algorithm, the results are shown in Fig.~\ref{fig:DynamicCar}. 

The unicycle  cannot follow the initial curve as seen in Fig.~\ref{fig:DynamicCar_s0_2d}. The AGHF yields the curve $x(t,s_{\max})$ shown in Fig. ~\ref{fig:DynamicCar_s05_2d} (black solid line). Extracting the control, we obtain a trajectory  (cyan dotted line) that is almost identical. 
\begin{figure}
\tikzset{every picture/.style={scale=0.58}}
\begin{subfigure}{.5\columnwidth}
  \centering
    \includegraphics{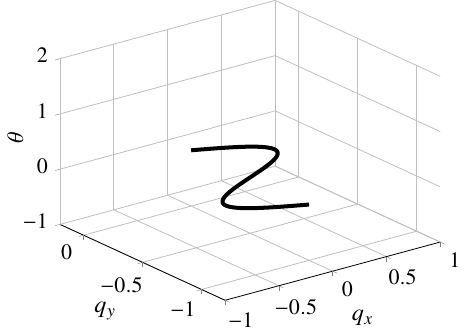}
  \caption{s=0}
  \label{fig:DynamicCar_s0_3d}
\end{subfigure}%
\begin{subfigure}{.5\columnwidth}
  \centering
  \includegraphics{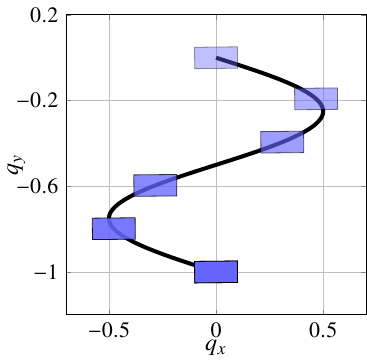}
  \caption{s=0}
  \label{fig:DynamicCar_s0_2d}
\end{subfigure}
\begin{subfigure}{.5\columnwidth}
  \centering
  \includegraphics{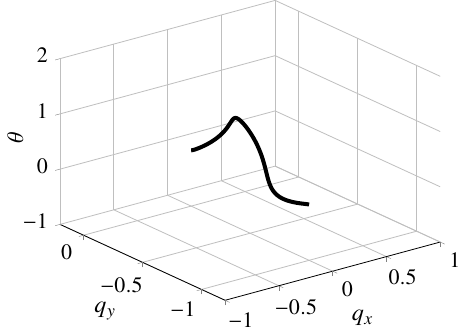}
  \caption{s=0.0005}
  \label{fig:DynamicCar_s0002_3d}
\end{subfigure}%
\begin{subfigure}{.5\columnwidth}
  \centering
    \includegraphics{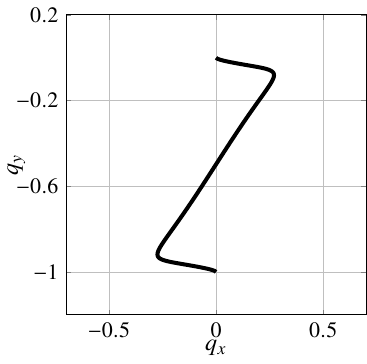}
  \caption{s=0.0005}
  \label{fig:DynamicCar_s0002_2d}
\end{subfigure}
\begin{subfigure}{.5\columnwidth}
  \centering
  \includegraphics{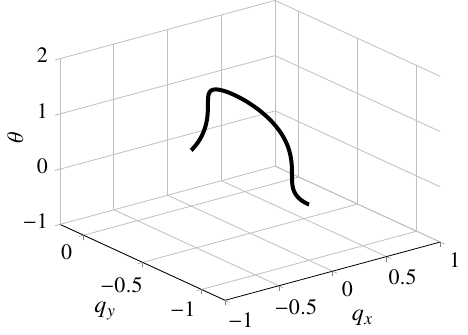}
  \caption{s=0.01}
  \label{fig:DynamicCar_s05_3d}
\end{subfigure}%
\begin{subfigure}{.5\columnwidth}
  \centering
  \includegraphics{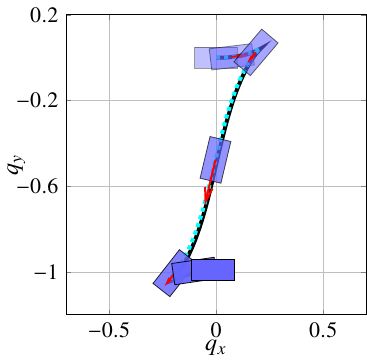}
  \caption{s=0.01}
  \label{fig:DynamicCar_s05_2d}
\end{subfigure}
\centering
\caption{\small Dynamic unicycle trajectory with $\lambda=50000, T=1$. Only $(q_x,q_y,\theta)$-space (left column) and $(q_x,q_y)$-projected view (right column) are shown here.} 
\label{fig:DynamicCar}
\end{figure}

\subsection{Unicycle with constrained inputs}
We now return to the original kinematics planar unicycle \eqref{eqn:unicycle}, and  we impose constraints on the input. Typical constraints include either the linear or steering velocity cannot be too large. Recall in Section~\ref{sec:constraints}, we have discussed that the first step for handling input constraints is to perform a dynamical extension and we will derive the equation \eqref{sys:dynamic_unicycle} again. We use the barrier function
\[
    b(x) = \frac{1}{(u_i^{\max})^2-{u_i}^2},\quad i=1\mbox{ or }2
\]
for the constraint in linear velocity or steering velocity, respectively. As a result, we have $L(x,\dot x)=b(x)L'(x,\dot x)$, where $L'$ is the old Lagrangian that we derived in \eqref{L_dynamic_unicycle}. It should be noted that while giving the boundary values and initial sketch of the states $u_1$ and $u_2$, they need to satisfy the constraints. The boundary values and initial sketch in the previous example with $u_1=u_2\equiv 0$ do always satisfy those requirements and we use them to solve \eqref{eq:aghf}. Thus we can follow the rest of the algorithm and derive an approximated path. The two cases of constrained linear velocity with $u_1^{\max}=2$ and constrained steering velocity with $u_2^{\max}=\frac{\pi}{2}$ individually are shown in Fig.~\ref{fig:DynamicCar_const}.
\begin{figure}
\tikzset{every picture/.style={scale=0.58}}
\begin{subfigure}{.5\columnwidth}
  \centering
    \includegraphics{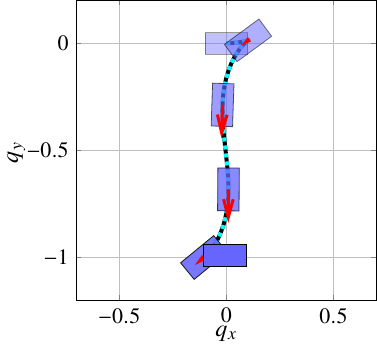}
  \caption{trajectory with constraint on linear velocity}
  \label{fig:DynamicCar_Z1const_2d}
\end{subfigure}%
\begin{subfigure}{.5\columnwidth}
  \centering
  \includegraphics{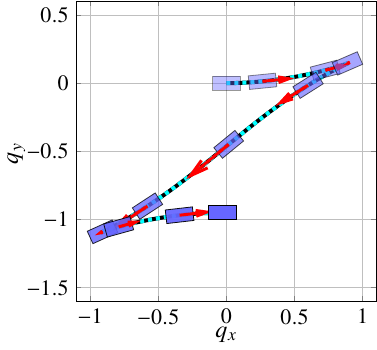}
  \caption{trajectory with constraint on steering velocity}
  \label{fig:DynamicCar_Z2const_2d}
\end{subfigure}
\begin{subfigure}{.5\columnwidth}
  \centering
  \includegraphics{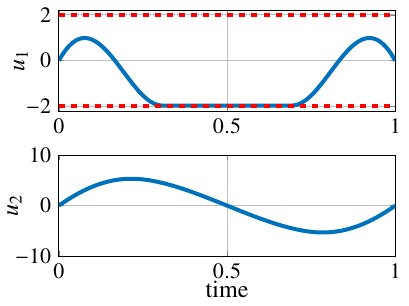}
  \caption{velocity profile with constraint on linear velocity}
  \label{fig:DynamicCar_Z1const_v}
\end{subfigure}%
\begin{subfigure}{.5\columnwidth}
  \centering
  \includegraphics{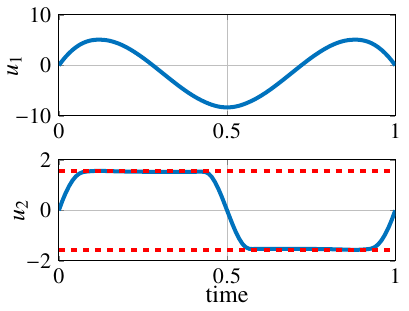}
  \caption{velocity profile with constraint on steering velocity}
  \label{fig:DynamicCar_Z2const_v}
\end{subfigure}
\centering
\caption{\small Dynamic unicycle trajectories with different constraints. When the linear velocity is constrained (left column), the integrated path tends to align with the straight line connecting $x_i,x_f$. More efforts are put in the steering in the beginning and ending of the trip in order to adjust the orientation fast as a compensation. When steering velocity is constrained, the integrated path is of distorted ``Z" shape. The peak value of the linear velocity is larger and the total length of the 2D path is much longer.}
\label{fig:DynamicCar_const}
\end{figure}
It is observed that for most $t\in[0,T]$, the constrained inputs almost meet their boundary values shown as the red dashed lines (however the constrained inputs will never reach their boundary values due to the soft barrier function $b(x)$); this is close to a bang-bang control strategy, which is very often the optimal control strategy for most control problems with bounded inputs. Nevertheless, with finite $\lambda$, our integrated path is always $C^1$ so the control is always continuous, thus we will never derive a true piece-wise continuous bang-bang control switching between boundary values. It remains an interesting yet challenging question that whether the limit of solutions when $\lambda$ goes to infinity is the ``true'' optimal path of minimal subRiemannian length.

\bibliographystyle{plain}
\bibliography{AGHFark}

\end{document}